\documentclass[a4paper]{amsart}
\usepackage{color}
\usepackage{graphicx,subfigure}

\newtheorem{theorem}{Theorem}[section]

\newtheorem{lemma}[theorem]{Lemma}

\theoremstyle{definition}
\newtheorem{definition}[theorem]{Definition}
\newtheorem{remark}[theorem]{Remark}
\newtheorem{example}[theorem]{Example}

\newcommand{\R}{\mathbb{R}}
\newcommand{\N}{\mathbb{N}}
\newcommand{\Z}{\mathbb{Z}}

\newcommand{\G}{\mathcal{G}}
\renewcommand{\O}{\mathcal{O}}
\newcommand{\F}{\mathcal{F}}
\newcommand{\D}{\partial}
\newcommand{\cl}[1]{\overline{#1}}
\newcommand{\wh}[1]{\widehat{#1}}

\newcommand{\dif}[1]{\,\mathrm{d}#1}
\newcommand{\ldif}[1]{\frac{\dif{\phantom{#1}}}{\dif{#1}}\,}

\DeclareMathOperator{\sign}{sign}

\newcommand{\MxN}{M\times N}

\newcommand{\f}{\mathfrak{f}}
\newcommand{\g}{\mathfrak{g}}
\newcommand{\h}{\mathfrak{h}}
\newcommand{\vv}{\mathfrak{v}}

\title[Forced oscillations]{Forced oscillations for generalized $\Phi$-Laplacian equations with Carath\'eodory perturbations}
\author[A.\ Calamai]{Alessandro Calamai}
\address{Alessandro Calamai, 
Dipartimento di Ingegneria Civile, Edile e Architettura,
Universit\`{a} Politecnica delle Marche
Via Brecce Bianche
I-60131 Ancona, Italy}%
\email{a.calamai@univpm.it}%
\author[M.P.\ Pera]{Maria Patrizia Pera}
\address{Maria Patrizia Pera,
Dipartimento di Matematica e Informatica ``Ulisse Dini'',
Universit\`a degli Studi di Firenze,
Via S.\ Marta 3, I-50139 Florence, Italy}%
\email{mpatrizia.pera@unifi.it}
\author[M.\ Spadini]{Marco Spadini}
\address{Marco Spadini,
Dipartimento di Matematica e Informatica ``Ulisse Dini'',
Universit\`a degli Studi di Firenze,
Via S.\ Marta 3, I-50139 Florence, Italy}%
\email{marco.spadini@unifi.it}
\thanks{The authors are members of the Gruppo Nazionale per l'Analisi Mate\-ma\-tica, la Probabilit\`a e le loro Applicazioni (GNAMPA) of the Istituto Nazionale di Alta Mate\-ma\-tica (INdAM), and
are supported by PRIN 2022 - Progetti di Ricerca di rilevante Interesse Nazionale, ``Nonlinear differential problems with applications to real phenomena'' (Grant Number: 2022ZXZTN2)}

\begin{document}
\begin{abstract}
Using topological methods, we study the structure of the set of forced oscillations of a class of parametric, implicit ordinary differential equations with a generalized $\Phi$-Laplacian type term.
We work in the  Carath\'eodory setting. Under suitable assumptions, involving merely the Brouwer degree in Euclidean spaces, we obtain global bifurcation results. In some illustrative examples we provide a visual representation of the bifurcating set.
\end{abstract}

 \maketitle

 \section{Introduction}

The purpose of this paper is to study the structure of the set of forced oscillations of a class of parametric 
ordinary differential equations (ODEs)
involving a generalized $\Phi$-Laplacian type differential operator.

More specifically, we consider the following two different problems which, although similar, have to be handled separately:
\begin{equation}\label{eq:0-pertb-intronew}
 [\phi\big(\lambda,x(t),x'(t)\big)]'=\lambda f\big(t,x(t),x'(t),\lambda\big),\quad \lambda\geq 0,
\end{equation}
and
\begin{equation}\label{eq:g-pertb-intronew}
 [\phi\big(\lambda,x(t),x'(t)\big)]'=g\big(x(t),x'(t)\big)+\lambda f\big(t,x(t),x'(t),\lambda\big),\quad \lambda\geq 0,
\end{equation}
where $\lambda$ is a real parameter and, for a given open set $U \subseteq \R^n$, the map $g\colon U\times\R^n\to\R^n$ is continuous, whereas the perturbing term $f\colon\R\times U\times\R^n\times[0,\infty)\to\R^n$ is Carath\'eodory and $T$-periodic in the first variable, $T>0$ being given. We study the set of periodic solutions that depend on the parameter $\lambda$. More precisely, see Definition \ref{def:Tforced}, we investigate the set of those pairs $(\lambda,x)$, which we call \emph{$T$-forced pairs} of \eqref{eq:0-pertb-intronew} (resp.\ \eqref{eq:g-pertb-intronew}), where $x$ is a $T$-periodic solution corresponding to $\lambda$ of \eqref{eq:0-pertb-intronew} (resp.\ \eqref{eq:g-pertb-intronew}).  
Our main goal will be to determine conditions ensuring the existence of a ``branch'', consisting of nontrivial (again see Definition \ref{def:Tforced}) $T$-forced pairs emanating for $\lambda=0$ from the set of stationary solutions.

Concerning the mapping $\phi\colon [0,\infty)\times U\times\R^n\to\R^n$ that appears on the left-hand side in \eqref{eq:0-pertb-intronew} and \eqref{eq:g-pertb-intronew}, it is assumed to be continuous and such that $\phi(\lambda,p,\cdot)$ is one-to-one and onto for each $\lambda\in[0,\infty)$ and $p\in U$.
Loosely speaking, we view $\phi(\lambda,p,q)$ as a sum $\phi_0(q)+\lambda \delta(\lambda,p,q)$, where $\phi_0$ is a homeomorphism and  $\delta(\lambda,p,q)$ is a suitable perturbation, see Section \ref{sec-main} for more details.

As a special case when  $n=1$ we recover the  classical $p$-Laplacian operator $\Phi(x):=x|x|^{p-2}$, with $p>1$, or, more generally, the so-called \emph{$\Phi$-Laplacian}: meaning a strictly increasing homeomorphism $\Phi\colon\R\to\R$. Scalar equations with $\Phi$-Laplacian, of the form 
\[
(\Phi(x'(t)))'= \mathfrak{k}(t,x(t),x'(t))
\]
are well-studied since they arise in some applicative models: for example, in non-Newtonian fluid theory,
nonlinear elasticity, diffusion of flows in porous media, theory of capillary surfaces
and, more recently, the modeling of glaciology
(see, e.g., \cite{BGKT,BogRon,PRRFB}).
{}From a mathematical viewpoint, different kinds of boundary conditions can be associated to these equations;
for instance, the existence of heteroclinic solutions can be established using the method of
 upper/lower solutions (see, e.g., \cite{BiCaPa19, cab11, Ca2011H, Cumapa2011, Frigon2013}).
 Concerning the study of periodic solutions, both in the scalar case and for equations in $\R^n$,
topological methods have been employed by many authors.
We cite, for instance, \cite{BeCa08, BeFe24, bm08, BFZ21, FSZ19, RT05}.

Recently, in \cite{capesp-libro} we investigated a similar problem but with the stronger condition that the map
$f$ in \eqref{eq:0-pertb-intronew} and \eqref{eq:g-pertb-intronew} is \emph{continuous}, and not merely Carath\'eodory.
Here we focus on the Carath\'eodory setting so that, with respect to \cite{capesp-libro},  we have to weaken the definition of solution of the ODEs \eqref{eq:0-pertb-intronew} and \eqref{eq:g-pertb-intronew}, cf.\ Definition~\ref{def:sol},
and as a tool we need to use some more recent results 
 obtained in \cite{CaSp24}.
We stress that our results cannot be deduced from the analogous ones in \cite{capesp-libro}.

The approach used in this paper to study the set of $T$-periodic solutions of~\eqref{eq:0-pertb-intronew} and \eqref{eq:g-pertb-intronew} can be summarized as follows. We observe that these two equations, by means of the introduction of an auxiliary variable, can be regarded as equivalent systems in $\R^{2n}$ written in normal form. These systems can be studied, in a similar way as in \cite{capesp-libro}, through earlier results, obtained by the authors with M.\ Furi, about periodically perturbed ODEs on differentiable manifolds in the Carath\'eodory setting, see 
\cite{FP97-car,Spa00}.
In fact here we exploit some very recent results on system of coupled equations on manifolds,
 obtained in \cite{CaSp24}, that in a certain sense extend and unify the aforementioned ones.
Let us point out that the cited results rely upon the topological notion
of degree (also called \textit{Euler characteristic}) of a tangent vector field, see the classical books
\cite{GyP, Hir76, MilTop}.
However, in the present paper all the technicalities are hidden in the proofs so that
our main results require only the well-known Brouwer degree in Euclidean spaces;
no advanced tools from differential topology are needed here.

Our main results, Theorems \ref{th:main1} and \ref{th:main2} below, 
can be described as \emph{global bifurcation results}, or \emph{atypical bifurcation results} in the sense of Prodi-Ambrosetti.
Roughly speaking, we obtain the existence of an unbounded connected set -- a ``branch'' -- made up of 
``nontrivial'' $T$-periodic solution pairs $(\lambda,x)$ of \eqref{eq:0-pertb-intronew} and \eqref{eq:g-pertb-intronew}, resp.,
that emanates from the set of the ``trivial'' ones
-- see Section \ref{sec-main} for more precise definitions. 
To prove the existence of such a branch we need only suitable assumptions on the Brouwer degree of some maps in $\R^n$, 
related with the terms $f$ and $g$ in the right-hand side of \eqref{eq:0-pertb-intronew}
and \eqref{eq:g-pertb-intronew}.
Notice that the Brouwer degree can be computed explicitly; therefore it is possible, in some concrete situations, to verify the validity of those assumptions.

We observe that our results are related to similar ones, concerning the set of $T$-periodic solutions of periodically perturbed ODEs: for the case of differential equations on manifolds we cite the already quoted papers \cite{CaSp24, FP97-car, Spa00}.
Concerning equations in $\R^n$, we mention the recent \cite{BeFe24} in which the authors study a case analogous to \eqref{eq:0-pertb-intronew} but with a different topological tool, namely the coincidence degree developed by J.\ Mawhin.

Moreover, our problems are linked to some class of
differential-algebraic equations, cf. \cite{bicasp, Ca2011B, CaSp2, Spa2010}. In fact, under suitable conditions, also in that case it is possible to relate
such equations to ODEs on manifolds, and to obtain the desired properties only in terms of the Brouwer degree in $\R^n$.

An interesting question, that we postpone to further research, is whether it would be possible to find a bridge between Theorems~\ref{th:main1} and \ref{th:main2}. In fact, Theorem \ref{th:coupled} does something similar for Theorems \ref{th:g-single} and \ref{th:0-single}, see Section \ref{sec:preliminari} and the pertinent discussion in~\cite{Spa2006}. Another attractive line of study, not addressed here, is the investigation of the set of $T$-periodic solutions of equations \eqref{eq:0-pertb-intronew} and \eqref{eq:g-pertb-intronew} when a dependence on delayed arguments is introduced in $\phi$ and $f$.

The paper is organized as follows. In Section \ref{sec:degree} we summarize the concept of Brouwer degree in Euclidean spaces in a slightly extended version. In Section \ref{sec-main} we list the main assumptions and state the global bifurcation results: namely, Theorems \ref{th:main1} and \ref{th:main2}. Section~\ref{sec:preliminari} is devoted to  some preliminary results about
periodic perturbations of autonomous ODEs on manifolds that we present here in a simplified form, suitable for our purposes: then, in Section \ref{se:rearrange}, we rewrite our equations in a way that allows the application of the results of the previous section.
Section \ref{se:proofs} contains the proof of the main results.

In the last Section \ref{sec:graphs} we present some purely illustrative examples, for which we show a possible way to  represent the set of $T$-forced pairs by means of suitable, finite-dimensional projections. More in detail, the visual representation of branches is obtained by plotting respectively, the $C^1$-norm of the solution $x$ in any $T$-forced pair $(\lambda,x)$, and the set of ``starting points'': meaning those triples $(\lambda,p,v)$ such that $(p,v)$ are the initial conditions, at time $t=0$, of a $T$-periodic solution that corresponds to $\lambda$.
In a sense, the numerical examples of Section \ref{sec:graphs} aim to illustrate different possible behaviors of the set of $T$-forced pairs.

\section{A quick refresher on the Brouwer degree in Euclidean spaces}\label{sec:degree}
In our argument the notion of Brouwer degree is only a tool whose use is hidden in the proofs. However, in order to make the paper self-contained, we briefly recall the definition and some properties since we will make use of this concept in a slightly extended version (see e.g.\ \cite{BFPS03,DiMa21,Ni74}). A reader which is already familiar with it can skip this section.

Let $\O$ be an open subset of $\R^s$, $\F$ a continuous $\R^s$-valued map, defined (at least) on the closure $\cl{\O}$ of $\O$, and $q \in \R^s$. We say that the triple $(\F,\O,q)$ is \emph{admissible (for the Brouwer degree)} if the set $\F^{-1}(q)\cap \O$ is compact.

The Brouwer degree is an integer-valued function which to any admissible triple $(\F, \O,q)$ assigns %
$\deg(\F, \O,q)\in\Z$, called the \emph{Brouwer degree of $\F$ in $\O$ with target $q$}.
One can say  informally that this number counts algebraically the solutions in $\O$ of the equation $\F(p) = q$. In fact, one of the properties of this integer-valued function is given by the following \emph{computation formula}, that holds in the smooth regular case. 

Recall that, given a map $\F\colon \O\to \mathcal \R^s$ of class $C\sp{1}$ and a point $p\in \O$, $p$ is said to be a \emph{regular point} (of $\F$) if the differential of $\F$ at $p$, $d\F_p$, is surjective. The points that are not regular are called \emph{critical points}.
The \emph{critical values} of $\F$ are those points of $\R^s$ which belong to the image of the set of the critical points.
Any $q\in \mathcal \R^s$ which is not in the image of the set of the critical points is a \emph{regular value}. So, $q$ is a regular value for $F$ in $\O$ if and only if $\det(d\F_p)\neq 0$, $\forall p \in\F^{-1}(q) \cap \O$. Note that any element of $\R^s$ which is not in the image of $\F$ is a regular value. The following formula is actually the basic definition of the Brouwer degree.

\medskip\noindent
\textbf{Computation formula.}
\label{Computation Formula}
If $(\F,\O,q)$ is admissible, $\F$ is smooth, and $q$ 
is a regular value for $\F$ in $\O$, then
\begin{equation}\label{eq:comp}
\deg(\F,\O,q)\; =\sum_{p \in\F^{-1}(q) \cap \O}  \sign \det (d\F_p).
\end{equation}

Let now $(\G,\O,r)$ be any admissible triple. Then, $\deg(\G,\O,r)$ is defined by taking $(\F,\O,q)$ such that $\F$ and $q$ are ``sufficiently close'' to $\G$ and $r$ and satisfy the assumptions of the Computation Formula and setting
\begin{equation}\label{eq:approx}
\deg(\G,\O,r) := \deg(\F,\O,q).
\end{equation}
It is known that this definition is well-posed.

We point out that the more classical, and well-known, definition of Brouwer degree is usually given in the subclass of triples $(\F,\O,q)$ such that $\F\colon\cl{\O}\to\R^s$ is continuous, $\O$ is bounded and $q \notin \F(\partial \O)$. However, all the standard properties of the degree, such as \emph{homotopy invariance, excision, additivity, existence}, still hold in this more general framework.
For a detailed list of such properties we refer, e.g., to \cite{Hir76,L78,Ni74}.

In this paper the target point $q$ will always be the origin of $\R^s$; therefore, for the sake of simplicity, we will adopt the shorter notation $\deg(\F,\O)$ instead of $\deg(\F,\O,0)$.
In this context, we will say that $p \in\F^{-1}(0)$ is a \emph{nondegenerate zero} (of $\F$) if $\det (d\F_p) \neq 0$; this means, equivalently, that $p$ is a regular point.
Accordingly, we will also say that $(\F,\O)$ is an \emph{admissible pair} (or that the map $F$ is admissible in $\O$) if so is the triple $(F,\O,0)$.
 Observe that $\deg(\F,\O)$ can be regarded also as the degree (or characteristic, or rotation) of the map $F$ seen as a tangent vector field on $\R^s$. This remark is crucial for the 
 extension of the above definitions to the context of differentiable manifolds (see, e.g., \cite{GyP, Hir76, MilTop}).

\section{Setting of the problem and statement of the main results}
\label{sec-main}

As pointed out in the Introduction, we study the following two equations:

\begin{equation}\label{eq:0-pertb-intro}
 [\phi\big(\lambda,x(t),x'(t)\big)]'=\lambda f\big(t,x(t),x'(t),\lambda\big),\quad \lambda\geq 0,
\end{equation}
and
\begin{equation}\label{eq:g-pertb-intro}
 [\phi\big(\lambda,x(t),x'(t)\big)]'=g\big(x(t),x'(t)\big)+\lambda f\big(t,x(t),x'(t),\lambda\big),\quad \lambda\geq 0.
\end{equation}

Throughout the paper we will make the following assumptions on the map $\phi$ that appears in \eqref{eq:0-pertb-intro} and \eqref{eq:g-pertb-intro}.

$\bullet$\  The map $\phi\colon [0,\infty)\times U\times\R^n\to\R^n$ is assumed to be continuous and such that $\phi(\lambda,p,\cdot)$ is one-to-one and onto for each $\lambda\in[0,\infty)$ and $p\in U$. 

The latter condition means that $\phi(\lambda,p,\cdot)$ is invertible with respect to the third variable in the following sense: for each $u \in\R^n$ there exist a unique $q\in \R^n$, depending on $(\lambda,p)$,  such that
\[
u=\phi(\lambda,p,q).
\]
We will denote this ``partial inverse'' as $\psi(\lambda,p,\cdot)$, so that $q=\psi(\lambda,p,u)$. 

In the following we will also denote by ``$\partial_i$'' the partial differentiation with respect to the $i$-th variable.
We assume that:

$\bullet$\  The map $\psi\colon [0,\infty)\times U\times\R^n\to\R^n$ defined above is continuous and, for each $u\in\R^n$ and $p\in U$, the map
$\lambda\mapsto \partial_1 \psi(\lambda,p,u)$
is also continuous.

Finally, we suppose that:

$\bullet$\  $\phi(0,\cdot,\cdot)$ depends only on the third variable; so that, setting $\phi_0(q):=\phi(0,p,q)$ thus we have $\psi(0,p,u)=\phi_0^{-1}(u)$.

This last assumption means that we regard $\phi(\lambda,\cdot,\cdot)$ as a perturbation of some (nonlinear) homeomorphism depending only on the third variable. With this in mind, we view $\phi(\lambda,p,q)$ as a sum $\phi_0(q)+\lambda \delta(\lambda,p,q)$, where $\phi_0$ is a homeomorphism and  $\delta(\lambda,p,q)$ is a perturbation such that the above assumptions are satisfied.

The set of assumptions on $\phi$ may seem complicated but in the more regular case they are quite straightforward, as shown in the following remark.
\begin{remark}\label{re:implicit}
When the  map $\phi$ is of class $C^1$, by the Implicit Function Theorem, we have that the above conditions involving the map $\psi$ are satisfied if, for any $(\lambda,p,q) \in[0,\infty)\times U\times\R^n$, the partial derivative $\D_3\phi(\lambda,p,q)$ is invertible.
\end{remark}


Hereafter by $L^1_T(\R)$ we denote the Banach space, isomorphic to $L^1((0,T),\R)$, of the $L^1_\textrm{loc}$ maps $\xi\colon\R\to\R$ that are $T$-periodic, in the sense that $\xi(t)=\xi(t+T)$ for a.e.\ $t\in\R$.
Recall that a map $f\colon\R\times U\times\R^n\times [0,\infty)\to\R^k$ as in \eqref{eq:0-pertb-intro} and \eqref{eq:g-pertb-intro} is said to be $T$-periodic  Carath\'eodory when:
\begin{itemize}
\item[(F1)] $f(t+T,p,q,\lambda)= f(t,p,q,\lambda)$, $\forall (p,q,\lambda)\in U\times\R^n\times [0,\infty)$ and for a.e.\ $t \in \R$;
\item[(F2)] $t\mapsto f(t,p,q,\lambda)$ is measurable, $\forall (p,q,\lambda)\in U\times\R^n\times [0,\infty)$;
\item[(F3)] $(p,q,\lambda)\mapsto f(t,p,q,\lambda)$ is continuous, for a.e.\ $t \in \R$;
\item[(F4)] for any compact set $K \subseteq U\times\R^n\times [0,\infty)$, there exists a nonnegative function $\sigma_K \in L^1_T(\R)$ such that
$|f(t,p,q,\lambda)|\leq\sigma_K(t)$, $\forall (p,q,\lambda)\in K$ and for a.e.\ $t \in \R$.
\end{itemize}

Let us now clarify the meaning of solution of \eqref{eq:0-pertb-intro} and \eqref{eq:g-pertb-intro}.

\begin{definition}\label{def:sol}
Let $I\subseteq \R$ be an interval, and $\lambda \geq 0$ be given. A function $x\in W_\mathrm{loc}^{1,1}(I)$ is said to be a solution of \eqref{eq:0-pertb-intro} if the function
\[
y\colon I \to \R^n, \quad y(t)=\phi \big(\lambda,x(t),x'(t)\big)
\]
 is in $W_\mathrm{loc}^{1,1}(I)$ too and the equality $y'(t)=\lambda f\big(t,x(t),x'(t),\lambda\big)$ holds for a.e.\ $t \in I$.\\
Analogously, a function $x\in W_\mathrm{loc}^{1,1}(I)$ is said to be a solution of \eqref{eq:g-pertb-intro} if the function $y$ as above is in $W_\mathrm{loc}^{1,1}(I)$ and the equality $y'(t)=g\big(x(t),x'(t)\big)+\lambda f\big(t,x(t),x'(t),\lambda\big)$ holds for a.e.\ $t \in I$.
\end{definition}

We denote by $C^r_T(U)$, $r\in\N\cup\{0\}$, the set of the $U$-valued, $T$-periodic, $C^r$ functions with the standard $C^r$ topology. For simplicity, when $r=0$, we write $C_T(U)$.

\begin{remark}\label{rem:regular}
Observe that when $x$ is a solution of \eqref{eq:0-pertb-intro} or of \eqref{eq:g-pertb-intro}, the function $y$ introduced in Definition \ref{def:sol} is in particular continuous being in $W_\mathrm{loc}^{1,1}(I)$. By the definition of $\psi$ we have $x'(t)=\psi\big(\lambda,x(t),y(t)\big)$ for all $t\in I$ so, $\psi$ being continuous, we find that $x$ is actually $C^1$.
\end{remark}

Remark \ref{rem:regular} justifies the following definition.

\begin{definition}\label{def:Tforced}
A pair $(\lambda,x)\in [0,\infty)\times C^1_T(U)$, such that $x$  is a $T$-periodic solution of \eqref{eq:0-pertb-intro} (resp.\ \eqref{eq:g-pertb-intro}), is said to be a \emph{$T$-forced pair} for \eqref{eq:0-pertb-intro} (resp.\ \eqref{eq:g-pertb-intro}). A $T$-forced pair $(\lambda,x)$ is called \emph{trivial} if $x$ is constant and $\lambda=0$. 
\end{definition}

In this paper we investigate the set of nontrivial $T$-forced pairs of \eqref{eq:0-pertb-intro} and \eqref{eq:g-pertb-intro}. Before stating our main results we need some further notation.
Let $\Omega$ be an open subset of $[0,\infty)\times C^1_T(U)$. Define the open subset $\Omega_U$ of $U$ as follows:
\[
\Omega_{U} = \big\{ p\in U: (0,\cl{p})\in\Omega\big\}.
\]
where  $\cl{p}$ denotes the constant function $p(t)=p $ for all $t\in \R$.

Concerning the set of $T$-forced pairs of equation \eqref{eq:0-pertb-intro} we will prove the following result:
\begin{theorem}\label{th:main1}
 Let $\Omega$ be an open subset of $[0,\infty)\times C^1_T(U)$. Define the vector field $w\colon U\to\R^n$ as follows:
 \[
 w(p):=\frac{1}{T}\int_0^T f(t,p,0,0)\dif{t},
 \]
 and assume that $\deg(w,\Omega_U)$ is well-defined and nonzero. Then there exists a connected set $\Gamma$ of nontrivial $T$-forced pairs in $\Omega$ of \eqref{eq:0-pertb-intro} whose closure in $[0,\infty)\times C^1_T(U)$ intersects the set $\big\{ (0,\cl{p})\in [0,\infty) \times C^1_T(U): p\in w^{-1}(0)\cap\Omega_U\big\}$ and is not contained in any compact subset of $\Omega$. In particular, when $U=\R^n$ and $\Omega=[0,\infty)\times C^1_T(\R^n)$ then $\Gamma$ is unbounded. 
\end{theorem}

Note that this theorem is not applicable to \eqref{eq:g-pertb-intro}: compare with Remark \ref{rem:differenza}. For this equation, instead, the pertinent result is as follows:

\begin{theorem}\label{th:main2}
 Let $\Omega$ be an open subset of $[0,\infty)\times C^1_T(U)$. Let $ \gamma(p):=g(p,0)$ and assume that  $\deg\big(\gamma,\Omega_U\big)$ is well-defined and nonzero. Then there exists a connected set $\Gamma$ of nontrivial $T$-forced pairs in $\Omega$ of \eqref{eq:g-pertb-intro} whose closure in $[0,\infty)\times C^1_T(U)$ intersects the set $\big\{ (0,\cl{p})\in [0,\infty)\times C^1_T(U): p\in \gamma^{-1}(0)\cap\Omega_U\big\}$ and is not contained in any compact subset of $\Omega$. In particular, when $U=\R^n$ and $\Omega=[0,\infty)\times C^1_T(\R^n)$ then $\Gamma$ is unbounded. 
\end{theorem}

We will prove these theorems by transforming equations \eqref{eq:0-pertb-intro} and \eqref{eq:g-pertb-intro} into a form suitable for investigation by the results of \cite{CaSp24}. We will also need a few technical facts about the topological degree of vector fields many of which carry over from~\cite{capesp-libro}. 

\begin{remark} \label{rem:differenza}
 Observe that the above Theorems \ref{th:main1} and \ref{th:main2}
have a similar statement and yield similar conclusions.
Nevertheless, it is not possible to consider one as a particular case of the other, even in the case when $g$ vanishes identically so that equation \eqref{eq:g-pertb-intro} reduces to \eqref{eq:0-pertb-intro}. 
To see this, notice that  the degree of the average wind $w$, that is crucial for Theorem \ref{th:main1}, plays no role in Theorem~\ref{th:main2}: in principle, it could not be even defined. Conversely, for equation \eqref{eq:0-pertb-intro}, the degree of $\gamma$ does not even make sense. 
\end{remark}

\section{Preliminary results}\label{sec:preliminari}
In order to pursue our investigation on the set of $T$-forced pairs of \eqref{eq:0-pertb-intro} and \eqref{eq:g-pertb-intro} we need some previous results about periodic perturbations of autonomous ODEs on manifolds, that have been first obtained in \cite{FuPe1997,FuSp1997,Spa2006} within the framework of continuous vector fields, and then extended in \cite{FP97-car,Spa00} in the Carath\'eodory setting.

In this section we present some recent facts on system of coupled equations on manifolds obtained in \cite{CaSp24}, that in a certain sense extend and unify the aforementioned ones. However, since in the context that we consider here it is not necessary to treat with equations on manifolds, we will present the indispensable results in a form simplified and significantly weaker but adequate to our purposes here.

Let $V,W\subseteq\R^k$ be open sets, and
consider the following system parametrized by $\lambda\geq 0$ on $V\times W$:
\begin{equation}\label{eq:mainCoupled}
\left\{
\begin{array}{l}
\dot\xi(t)=\lambda \f\big(t,\xi(t),\eta(t),\lambda\big),\\
\dot\eta(t)=\g\big(\xi(t),\eta(t)\big)+\lambda \h\big(t,\xi(t),\eta(t),\lambda\big).
\end{array}
\right.
\end{equation}
where $\f\colon\R\times V\times W\times [0,\infty)\to\R^k$, $\h\colon\R\times V\times W\times [0,\infty)\to\R^k$ are Carath\'eodory, $\g\colon V\times W\to\R^k$ is continuous, and $\f$ and $\h$ are $T$-periodic in the first variable.

\begin{definition} \label{def:T-triple}
 A triple $(\lambda,\xi,\eta)\in [0,\infty)\times C_T(V\times W)$, such that \eqref{eq:mainCoupled} holds for a.e.\ $t$, is called a \emph{$T$-triple} for \eqref{eq:mainCoupled}. A $T$-triple $(\lambda,\xi,\eta)$ is \emph{trivial} if $(\xi,\eta)$ is constant and $\lambda=0$.
\end{definition}

Given $(p,q)\in V\times W$, it is convenient to introduce some further notation. Recall that by $\cl{p}$ and $\cl{q}$ we mean the functions constantly equal to $p$ and $q$, respectively. According to this stipulation, a $T$-triple is trivial if and only if it is of the form $(0,\cl{p},\cl{q})$ with  $(p,q)\in\g^{-1}(0)$.

We are in a position to state the following key result concerning the set of $T$-triples of \eqref{eq:mainCoupled}, Theorem~\ref{th:coupled} below. It is indeed a particular case of Theorem 3.1 of \cite{CaSp24} where systems are considered on $\MxN$, the product of manifolds $M$ and $N$ of possibly (also trivial) different dimensions.
For this reason the proof is omitted.

\begin{theorem}[\cite{CaSp24}]\label{th:coupled}
Let $\f$, $\g$ and $\h$ be as in equation \eqref{eq:mainCoupled}, and let $\O$ be an open subset of $[0,\infty)\times C_T(V\times W)$. Put $\O_{V\times W}=\big\{(p,q)\in V\times W:(0,\cl{p},\cl{q})\in\O\big\}$ and let $\nu\colon V\times W\to\R^{k}\times\R^{k}$ be the vector field given by
\[
\nu(p,q)=\left(\frac{1}{T}\int_0^T\f(t,p,q,0)\dif{t}\;,\;\g(p,q)\right).
\]
Assume that $\deg\big(\nu,\O_{V\times W}\big)$ is well-defined and nonzero. Then there exists a connected set $\Upsilon$ of nontrivial $T$-triples in $\O$ of \eqref{eq:mainCoupled} whose closure in $[0,\infty)\times C_T(V\times W)$ intersects the set $\big\{(0,\cl{p},\cl{q})\in[0,\infty)\times C_T(V\times W):(p,q)\in\nu^{-1}(0)\cap\O_{V\times W}\big\}$ and is not contained in any compact subset of $\O$. In particular, if $V=W=\R^k$ and $\O=[0,\infty )\times C_T(\R^{2k})$, then $\Upsilon$ is unbounded.
\end{theorem}

Following \cite[Remark 3.2]{CaSp24}, when system \eqref{eq:mainCoupled} is decoupled (for simplicity here we take $V=W$), we obtain results similar to  those of \cite{FuPe1997,Spa00}. Consider the following parameter-dependent differential equations on the open set $V\subseteq\R^k$, with $\lambda\in[0,\infty)$:
\begin{equation}\label{eq:g-single}
 \dot\xi(t)=\g\big(\xi(t)\big)+\lambda \f\big(t,\xi(t),\lambda\big),
\end{equation}
and
\begin{equation}\label{eq:0-single}
 \dot\xi(t)=\lambda \f\big(t,\xi(t),\lambda\big),
\end{equation}
where $\g\colon V\to\R^k$ is continuous and $\f\colon\R\times V\times [0,\infty)\to\R^k$ is Carath\'eodory and $T$-periodic in the first variable.

\begin{definition}
A pair $(\lambda,\xi)\in [0,\infty)\times C_T(V)$, such that \eqref{eq:g-single} (resp.\ \eqref{eq:0-single}) holds identically is a \emph{$T$-pair} for \eqref{eq:g-single} (resp.\ \eqref{eq:0-single}). A $T$-pair $(\lambda,\xi)$ is called \emph{trivial} if $\xi$ is constant and $\lambda=0$. 
\end{definition}

Given any point $p\in V$, as above, we denote by $\cl{p}$ the function constantly equal to $p$. Thus, a $T$-pair for \eqref{eq:g-single} is trivial if and only if it is of the form $(0,\cl{p})$ for some $p\in\g^{-1}(0)$. Clearly, all pairs $(0,\cl{p})$ are trivial $T$-pairs for \eqref{eq:0-single}, for all $p\in V$. 

As it is the case for equation \eqref{eq:mainCoupled}, qualitative properties of the set of $T$-pairs of \eqref{eq:g-single} and \eqref{eq:0-single} can be obtained from the degree of suitably chosen vector fields \cite[Thm.\ 3.1 and Remark 3.2]{CaSp24} compare also \cite{FuPe1997,Spa00}.

\begin{theorem}\label{th:g-single}
Let $\f$ and $\g$ be as in equation \eqref{eq:g-single} and let $\O $ be an open subset of $[0,\infty)\times C_T(V)$. Put $\O_V=\big\{p\in V:\cl{p}\in\O\big\}$. Assume that $\deg (\g,\O_V)$ is well defined and nonzero. Then there exists a connected set $\Upsilon$ of nontrivial $T$-pairs in $\O$ whose closure in $\left[ 0,\infty \right) \times C_T(V)$ intersects the set $\big\{(0,\cl{p})\in[0,\infty)\times C_T(V):p\in\g^{-1}(0)\cap\O_V\big\}$ and is not contained in any compact subset of $\O$. In particular, if $V=\R^k$ and $\O =\left[0,\infty \right) \times C_T(\R^k)$, then $\Upsilon$ is unbounded.
\end{theorem}

Similarly,

\begin{theorem}\label{th:0-single}
Let $\f$ be as in \eqref{eq:0-single}, define $\vv\colon V\to\R^k$ be the autonomous vector field given by
\[
 \vv(p):=\frac{1}{T}\int_0^T \f(t,p,0)\dif{t}.
\]
Let $\O$ be an open subset of $[0,\infty)\times C_T(V)$, and put $\O_V=\big\{p\in V:\cl{p}\in\O\big\}$. Assume that $\deg (\vv,\O_V)$ is well defined and nonzero. Then there exists a connected set $\Upsilon$ of nontrivial $T$-pairs in $\O$ whose closure in $\left[ 0,\infty \right) \times C_T(V)$ intersects the set $\big\{(0,\cl{p})\in[0,\infty)\times C_T(V): p\in \vv^{-1}(0)\cap\O_V\big\}$ and is not contained in any compact subset of $\O$. In particular, if $V=\R^k$ and $\O =\left[0,\infty \right) \times C_T(\R^k)$, then $\Upsilon$ is unbounded.
\end{theorem}

Observe that Theorem \ref{th:0-single} concerns the case of perturbations of the zero vector field, which is not covered in general by Theorem~\ref{th:g-single}. In fact, under the above assumptions,
Theorem \ref{th:g-single} above is not applicable, $\g$ not being admissible.

 \section{Rearrangement of equations \eqref{eq:0-pertb-intro} and \eqref{eq:g-pertb-intro}}\label{se:rearrange}
 
 We now rewrite equations \eqref{eq:0-pertb-intro} and \eqref{eq:g-pertb-intro} into a form suitable for the application of the results of the previous section.
More precisely,
it is convenient to rewrite equations \eqref{eq:0-pertb-intro} and \eqref{eq:g-pertb-intro} as systems of ODEs in normal form. To do so, suppose that $x$ is a solution of \eqref{eq:0-pertb-intro} and introduce the function $y(t):=\phi\big(\lambda,x(t),x'(t)\big)$
which belongs to $W_\mathrm{loc}^{1,1}(I)$. Thus $x'(t) = \psi\big(\lambda,x(t),y(t)\big)$
and
\[
y'(t)=\lambda f\Big(t,x(t),\psi\big(\lambda,x(t),y(t)\big),\lambda\Big)
\]
so that the function $t\mapsto \big(x(t),y(t)\big)$, taking values in $U\times\R^n$, is in  $W_\mathrm{loc}^{1,1}(I)$ and satisfies the system:
\begin{equation}\label{eq:sysequiv}
\left\{
\begin{array}{l}
x'(t)=\psi\big(\lambda,x(t),y(t)\big)\\
y'(t)=\lambda f\Big(t,x(t),\psi\big(\lambda,x(t),y(t)\big),\lambda\Big),
\end{array}
\right.
\end{equation}
for a.e.\ $t \in \R$. In other words $(x,y)$ is a solution of \eqref{eq:sysequiv}.

In a similar way, if $x$ is a solution of \eqref{eq:g-pertb-intro} we find that $t\mapsto \big(x(t),y(t)\big)$ satisfies
\begin{equation}\label{eq:sysequiv2}
\left\{
\begin{array}{l}
x'(t)=\psi\big(\lambda,x(t),y(t)\big)\\
y'(t)=g\Big(x(t),\psi\big(\lambda,x(t),y(t)\big)\Big)+\lambda f\Big(t,x(t),\psi\big(\lambda,x(t),y(t)\big),\lambda\Big),
\end{array}
\right.
\end{equation}
for a.e.\ $t \in \R$, that is $(x,y)$ is a solution of \eqref{eq:sysequiv2}.

Conversely, consider a solution $t\mapsto \big(x(t),y(t)\big)\in U\times\R^n$ of \eqref{eq:sysequiv} (resp.\ \eqref{eq:sysequiv2}) and note that, by definition, $(x,y)\colon J\to\R^n\times\R^n$ is in $W_\mathrm{loc}^{1,1}(J)$, $J\subset\R$ being an interval.  By the first equation of \eqref{eq:sysequiv} (resp.\ \eqref{eq:sysequiv2}) we get $y(t):=\phi\big(\lambda,x(t),x'(t)\big)$, thus we have that
$x$ is a solution of \eqref{eq:g-pertb-intro} (resp.\ \eqref{eq:0-pertb-intro}).

Our purpose now is to rewrite equations \eqref{eq:sysequiv} and \eqref{eq:sysequiv2} into a more useful form by ``isolating'' $\lambda$ from inside $\psi$ as well as $g$ in these systems. A crucial step is the following simple  version of Hadamard's Lemma, compare e.g.\ \cite{Arn2006}:
 \begin{lemma}[\cite{capesp-libro}]\label{le:hadamard1}
 Let $\phi$ be as in equations \eqref{eq:0-pertb-intro} and \eqref{eq:g-pertb-intro}.  Let also $p\in U$, $q,u\in\R^n$ and $\lambda\in[0,\infty)$ be such that $u=\phi(\lambda,p,q)$. Then there exists a continuous function $h$ such that  $q=\phi_0^{-1}(u)+\lambda h(\lambda,p,u)$. 
 \end{lemma}
 \begin{proof}
 Put $q=\psi(\lambda,p,u)$. It is enough to observe that
 \begin{align*}
  q-\phi_0^{-1}(u)=
  q-\psi(0,p,u)=&\int_0^1 \ldif{s}\psi(s\lambda,p,u)\dif{s}\\
  =&\int_0^1 \lambda \partial_1\psi(s\lambda,p,u)\dif{s}=
  \lambda h(\lambda,p,u),
 \end{align*}
where $h(\lambda,p,u)=\int_0^1\partial_1\psi(s\lambda,p,u)\dif{s}$.
\end{proof}

 By Lemma \ref{le:hadamard1} we see that 
 \begin{equation}\label{eq:psi-ident}
 \psi\big(\lambda,x(t),y(t)\big)=\phi_0^{-1}\big(y(t)\big)+\lambda h\big(\lambda,x(t),y(t)\big)
 \end{equation}
 for some mapping $h$. Hence system \eqref{eq:sysequiv} can be rewritten as
 \begin{equation}\label{eq:equiv1}
  \left\{
\begin{array}{l}
x'(t)=\phi_0^{-1}\big(y(t)\big)+\lambda h\big(\lambda,x(t),y(t)\big)\\
y'(t)=\lambda f\Big(t,x(t),\psi\big(\lambda,x(t),y(t)\big),\lambda\Big),
\end{array}
\right.
 \end{equation}

 Let us now deal with system \eqref{eq:sysequiv2}. Using \eqref{eq:psi-ident} we obtain
 \[
  g\Big(x(t),\psi\big(\lambda,x(t),y(t)\big)\Big)=g\Big(x(t),\phi_0^{-1}\big(y(t)\big)+\lambda h(\lambda,x(t),y(t)\big)\Big),
 \]
and using again Lemma \ref{le:hadamard1}, we get, for an appropriate function $\wh h$, the equality:
\begin{equation*}
 g\Big(x(t),\phi_0^{-1}\big(y(t)\big)+\lambda h(\lambda,x(t),y(t)\big)\Big)
 =g\Big(x(t),\phi_0^{-1}\big(y(t)\big)\Big)+\lambda\wh h(\lambda,x(t),y(t)\big).
\end{equation*}
Thus we can rewrite \eqref{eq:sysequiv2} as
\begin{equation}\label{eq:equiv2}
  \left\{\begin{array}{l}
         x'(t)=\phi_0^{-1}\big(y(t)\big)+\lambda h(\lambda,x(t),y(t)\big),\\
         y'(t)=g\Big(x(t),\phi_0^{-1}\big(y(t)\big)\Big)+\lambda F\big(t,x(t),y(t),\lambda\big).
         \end{array}
         \right.
 \end{equation}
where $F$ collects all the $\lambda$-dependent terms in the second equation of \eqref{eq:sysequiv2}, that is:
\begin{equation}\label{eq:bigF}
 F\big(t,x(t),y(t),\lambda\big)=\wh h(\lambda,x(t),y(t)\big)+f\Big(t,x(t),\psi\big(\lambda,x(t),y(t)\big),\lambda\Big).
\end{equation}

 \section{Proof of the main results}
 \label{se:proofs}

 \subsection{Proof of Theorem \ref{th:main1}}
 \label{se:proofThm1}

This section is devoted to the proof of Theorem \ref{th:main1} regarding the structure of the set of solutions of of \eqref{eq:0-pertb-intro}. As a first step, we relate the set $X \subseteq [0,\infty)\times C^1_T(U)$ of $T$-forced pairs of \eqref{eq:0-pertb-intro} with the set $Y\subseteq [0,\infty)\times C_T(U\times\R^n)$ of $T$-triples of~\eqref{eq:equiv1}, see  Definition \ref{def:T-triple}. To do so, we define the map $\mathfrak{H}\colon Y\to X$ by $\mathfrak{H}\colon(\lambda,x,y)\mapsto (\lambda,x)$, and observe that $\mathfrak{H}$ is continuous.

Lemma \ref{lem:homomT} below is analogous to Lemma 7.1 in \cite{capesp-libro}. We give the proof for completeness.

 \begin{lemma} \label{lem:homomT}
 The map $\mathfrak{H}$ is a homeomorphism that respects the notion of triviality, in the sense that it makes 
  trivial $T$-triples correspond to trivial $T$-forced pairs and vice versa.
 \end{lemma}

\begin{proof}
Let $(\lambda,x,y)$ be a $T$-triple of \eqref{eq:equiv1}; then, by definition, $x$ is a $T$-periodic solution of \eqref{eq:0-pertb-intro} corresponding to $\lambda$. Now, if $(\lambda,x,y)$ is trivial, this means that $\lambda=0$ and $x$ is constant: therefore $(\lambda,x)$ is a trivial $T$-forced pair. 
  
Conversely, let $(\lambda,x)$ be a $T$-forced pair of \eqref{eq:0-pertb-intro}; this means, as already observed in Remark \ref{rem:regular}, that $x$ is $C^1$ and, by the continuity of $\phi$, the map $t\mapsto y(t)=\phi\big(\lambda,x(t),x'(t)\big)$ is continuous. Note that $y$ is $T$-periodic and that $(x,y)$ is a $T$-periodic solution of \eqref{eq:equiv1} corresponding to $\lambda$. Namely, 
\[
  \mathfrak{H}^{-1}(\lambda,x) = \left(\lambda,x, \phi\big(\lambda,x,x'\big)\right) = (\lambda,x,y).
\]
It is not difficult, using the continuity of $\phi$, to show that $\mathfrak{H}^{-1}\colon X\to Y$ is a continuous map.

Finally, if $x$ is constant so is $y$, hence the image under $\mathfrak{H}^{-1}$ of any trivial $T$-forced pair is a trivial $T$-triple.
 \end{proof}

 Consider the following vector field, defined on $U\times \R^{n}$:
 \[
 \nu(p,u):=\left(\frac{1}{T}\int_0^T f\big(t,p,\phi_0^{-1}(u),0\big)\dif{t}\;,\;\phi_0^{-1}(u)\right).
 \]
In view of Theorem \ref{th:coupled} and Lemma \ref{lem:homomT} one sees that the degree of $\nu$ plays a crucial role in our argument. However, as shown by the following lemma, one does not need to compute the degree of $\nu$. In fact, since $\phi_0^{-1}$ is a homeomorphism, the computation of the degree of $\nu$ can be reduced to that of the ``average wind'' i.e.\ 
\begin{equation}\label{def:w}
 w(p):=\frac{1}{T}\int_0^T f(t,p,0,0)\dif{t}.
\end{equation}

Again, Lemma \ref{lem:redDeg1} below is similar to Lemma 7.2 in \cite{capesp-libro}.
We sketch the proof for the reader's convenience.

\begin{lemma}[\cite{capesp-libro}]\label{lem:redDeg1}
 Let $W\subseteq U$ be open. The vector field $\nu$ is admissible for the degree in $W\times\R^n$ if and only if so is $w$ in $W$, and 
\begin{equation}\label{eq:fdeg}
  \deg\big(\nu,W\times\R^n\big)=\pm \deg\big(w,W\big).
\end{equation}
\end{lemma}

\begin{proof}
To prove the first assertion, let $u_0\in\R^n$ be the unique point such that $\phi_0^{-1}(u_0)=0$ and observe that
\begin{equation}\label{eq:nuzero}
 \nu^{-1}(0) = \big\{(p,u_0):p\in w^{-1}(0)\big\};
\end{equation}
thus, $w^{-1}(0)\cap W$ is compact if and only if so is $\nu^{-1}(0)\cap(W\times\R^n)$. Hence $\nu$ is admissible in $W\times\R^n$ if and only if so is $w$ in $W$. 

Now, the proof of formula \eqref{eq:fdeg} follows from an application of the properties of the degree.
First, by the excision property, we get
\begin{equation}\label{eq:nu-exc}
 \deg\big(\nu,W_1\times\R^n\big)=\deg\big(\nu,W_1\times W_2\big)
\end{equation}
where $W_1 \subseteq W$ and $W_2\subseteq \R^n$ are open and bounded sets such that $\nu^{-1}(0)\subseteq W_1\times W_2$.
Let now $\nu_1$ be the first component of $\nu$. By known transversality theorems, we can approximate $\nu_1$ with a smooth map $\hat\nu_1$ with only isolated zeros and also approximate $\phi_0$ with a smooth diffeomorphism $\hat\phi_0$. Let
\[
 \hat\nu(p,u)=\big(\hat\nu_1(p,u)\;,\;\hat\phi_0^{-1}(u)\big).
\]
We can choose the approximations so close, in a way that the homotopy
\begin{multline*}
 H(s,p,u)=s\hat\nu(p,u)+(1-s)\nu(p,u) \\
 =\Big( s\hat\nu_1(p,u)+(1-s)\nu_1(p,u)\;,\; s\hat\phi_0^{-1}(u)+(1-s) \phi_0^{-1}(u)\Big),
\end{multline*}
has no zeros on the boundary of the bounded open set $W_1\times W_2$ for $s\in [0,1]$. Thus $H$ is admissible in $W_1\times W_2$, and the homotopy property of the degree yields
\begin{equation}\label{eq:nuhom}
 \deg\big(\nu,W_1\times W_2\big)=\deg\big(\hat\nu,W_1\times W_2\big).
\end{equation}
Similarly, defining $\hat w(p)=\hat\nu_1(p,0)$, we have
\begin{equation}\label{eq:whom}
 \deg(w,W_1)=\deg(\hat w,W_1).
\end{equation}

To conclude, let $(p_0,u_0)$ be an isolated zero of $\hat\nu$. We have
\begin{equation*}
 \det\hat\nu'(p_0,u_0)=\det\begin{pmatrix}
               \D_1\hat\nu_1(p_0,u_0) & 0 \\
               \D_2\hat\nu_1(p_0,u_0) & (\hat\phi_0^{-1})'(u_0)
               \end{pmatrix}
                =\det \hat w'(p_0)\det (\phi_0^{-1})'(u_0).
\end{equation*}
As in \eqref{eq:nuzero} we have 
\[
 \hat\nu^{-1}(0) = \big\{(p,u_0):p\in \hat w^{-1}(0)\big\}.
\]
Also, $\hat\phi_0^{-1}$ being a diffeomorphism,  $\sign\det (\hat\phi_0^{-1})'(u_0)=\pm 1$.
Thus,
\begin{align*}
\deg(\hat\nu,W_1\times W_2) 
  &= \sum_{(p_0,u_0)\in\hat\nu^{-1}(0)\cap W_1\times W_2}\sign\det\hat\nu'(p_0,u_0)\\
  &= \sign\det (\hat\phi_0^{-1})'(u_0) \sum_{p_0\in \hat w^{-1}(0)\cap W_1}\sign\det \hat w'(p_0)\\
  &= \pm \deg(\hat w,W_1).
\end{align*}

The assertion now follows from equations \eqref{eq:whom}, \eqref{eq:nuhom} and \eqref{eq:nu-exc}.
\end{proof}

\begin{proof}[Proof of Theorem \ref{th:main1}]
Recall that the standard $C^1$ topology of $C^1(\R^n)$ is finer than the one induced by the inclusion in $C(\R^n)$. Thus every open set of $C^1(U)$ is open in $C(U)$ as well. By the same argument we have that every open set of $[0,\infty)\times C_T^1(U)$ is also open in $[0,\infty)\times C_T(U)$. Consequently, since $\Omega$ is an open subset of $[0,\infty)\times C^1_T(U)$, the set $\Omega^*:=\Omega\times C_T(\R^n)$ is open in $[0,\infty)\times C_T(U)\times C_T(\R^n)$ that, by abuse of notation, we identify with $[0,\infty)\times C_T(U\times\R^n)$.

 Similarly to $\Omega_{U} = \big\{ p\in U: (0,\cl{p})\in\Omega\big\}$ that is clearly open in $U$, we construct the open set $\Omega^*_{U\times\R^n}\subseteq U\times\R^n$ by
\[
 \Omega^*_{U\times\R^n} = \big\{ (p,q)\in U\times\R^n: (0,\cl{p},\cl{q})\in\Omega^*\big\}= \Omega_{U}\times\R^n.
\]

Lemma \ref{lem:redDeg1} implies that $\deg(w,\Omega_U)$ is well defined and nonzero if and only if so is $\deg(\nu,\Omega^*_{U\times\R^n})$. Thus, by Theorem \ref{th:coupled} applied to \eqref{eq:equiv1} we find a connected set $\Theta$ of nontrivial $T$-triples in $\Omega^*$ of \eqref{eq:equiv1} whose closure in $[0,\infty)\times C_T(U\times\R^n)$ meets $\nu^{-1}(0)\cap\Omega^*_{U\times\R^n}$ and is not contained in any compact subset of $\Omega^*_{U\times\R^n}$. 

According to Lemma \ref{lem:homomT}, $\Gamma=\mathfrak{H}(\Theta)$ is a set of $T$-forced pairs for \eqref{eq:0-pertb-intro} in $\Omega$ such that trivial $T$-triples  of \eqref{eq:equiv1} correspond to trivial $T$-forced pairs. Since $\mathfrak{H}$ is a homeomorphism, the assertion follows.
\end{proof}

\begin{remark}\label{re:extension1}
 The techniques shown so far actually allow for an extension of Theorem \ref{th:main1} to a slight generalization of equation \eqref{eq:0-pertb-intro} as follows:
 \begin{equation}\label{eq:0-pertb-gen}
 \Big[\phi\Big(\lambda,x(t),x'(t)+\lambda k\big(t,x(t)\big)\Big)\Big]'=\lambda f\big(t,x(t),x'(t),\lambda\big),\quad \lambda\geq 0,
\end{equation}
with $\phi$ and $f$ as in \eqref{eq:0-pertb-intro} and $k\colon\R\times U\to\R^n$ Carath\'eodory and $T$-periodic in $t$. In fact, the construction of Section \ref{se:rearrange} that leads to \eqref{eq:equiv1} can be repeated here taking us to the following system equivalent to \eqref{eq:0-pertb-gen}:
\begin{equation}\label{eq:equiv1-gen}
  \left\{
\begin{array}{l}
x'(t)=\phi_0^{-1}\big(y(t)\big)+\lambda H\big(\lambda,x(t),y(t)\big)\\
y'(t)=\lambda f\Big(t,x(t),\psi\big(\lambda,x(t),y(t)\big),\lambda\Big),
\end{array}
\right.
 \end{equation}
 where 
\begin{equation}\label{eq:bigH}
 H\big(\lambda,x(t),y(t)\big)=h(\lambda,x(t),y(t)\big)- k\big(t,x(t)\big).
\end{equation}
The notion of $T$-forced pairs for \eqref{eq:0-pertb-gen} and $T$-triples of \eqref{eq:equiv1-gen} are the obvious counterparts of the analogous concepts for \eqref{eq:0-pertb-intro} and \eqref{eq:equiv1}. Clearly, as in Lemma \ref{lem:homomT} one finds a homeomorphism between the set of $T$-forced pairs of \eqref{eq:0-pertb-gen} and the $T$-triples of \eqref{eq:equiv1-gen}. Observe also that the perturbing term on the first equation of \eqref{eq:equiv1-gen} do not play any actual r\^ole in  the proof of Theorem \ref{th:main1}, thus it can be immediately checked that, when $\deg(w,\Omega_U)\neq 0$, the same argument of the theorem yields the existence of a connected set $\Gamma$ of nontrivial $T$-forced pairs in $\Omega$ of \eqref{eq:0-pertb-gen} whose closure in $[0,\infty)\times C^1_T(U)$ intersects the set $\big\{ (0,\cl{p})\in [0,\infty) \times C^1_T(U): p\in w^{-1}(0)\cap\Omega_U\big\}$ and is not contained in any compact subset of $\Omega$. In particular, when $U=\R^n$ and $\Omega=[0,\infty)\times C^1_T(\R^n)$ then $\Gamma$ is unbounded. 
\end{remark}

\subsection{Proof of Theorem \ref{th:main2}.}\label{se:proofThm2}

 As in the previous Section \ref{se:proofThm1}, we relate the set $\mathfrak{X} \subseteq [0,\infty)\times C_T^1(U)$ of $T$-forced pairs of \eqref{eq:g-pertb-intro} with the set $\mathfrak{Y}\subseteq [0,\infty)\times C_T(U\times\R^n)$ of $T$-triples of \eqref{eq:equiv2}. To do so, define the map $\mathfrak{G}\colon \mathfrak{Y}\to \mathfrak{X}$ by $\mathfrak{G}\colon(\lambda,x,y)\mapsto (\lambda,x)$. A result analogous to Lemma \ref{lem:homomT} holds in this case too. Since it follows from the same argument as Lemma \ref{lem:homomT} we omit the proof.

 \begin{lemma}\label{lem:homomTg}
  $\mathfrak{G}$ is a homeomorphisms that respects the notion of triviality, in the sense that it makes 
  trivial $T$-triples correspond to trivial $T$-forced pairs and vice versa.
 \end{lemma}

 Unlike Section \ref{se:proofThm1} where the assertion of Theorem  \ref{th:main1} was deduced from Theorem~\ref{th:coupled} applied to \eqref{eq:equiv1}, here we apply Theorem \ref{th:g-single} to \eqref{eq:equiv2}. To do so, we introduce the vector field $G$ defined on $U\times \R^{n}$ given by
 \[
G(p,u)=\Big(\phi_0^{-1}(u)\;,\; g\big(p,\phi_0^{-1}(u)\big)\Big).
\]
As in Section \ref{se:proofThm1}, since $\phi_0^{-1}$ is a homeomorphism, the degree of $G$ reduces up to its sign, to the degree of the simpler vector field $\gamma\colon U\to\R^n$ given by
\begin{equation}\label{def:omega}
 \gamma(p):=g(p,0).
\end{equation}
In other words, we have

\begin{lemma}\label{lem:redDeg2}
 Let $W\subseteq U$ be open. The vector field $G$ is admissible for the degree in $W\times\R^n$ if and only if so is $\gamma$ in $W$, and 
\begin{equation}\label{eq:fdeg1}
  \deg\big(G,W\times\R^n\big)=\pm \deg\big(\gamma,W\big).
\end{equation}
\end{lemma}
The proof of this assertion can be carried out as in Lemma \ref{lem:redDeg1} and therefore it is omitted.

The argument for proving Theorem \ref{th:main2} proceeds in a similar way to that of Theorem \ref{th:main1} performed in Section \ref{se:proofThm1}.

\begin{proof}[Proof of Theorem \ref{th:main2}]
As pointed out in the proof of Theorem \ref{th:main1}, the fact that $\Omega\subseteq[0,\infty)\times C^1_T(U)$ is open implies that the sets
\begin{align*}
&\Omega_{U} = \big\{ p\in U: (0,\cl{p})\in\Omega\big\}\subseteq U \subseteq \R^n,\\
&\Omega^*=\Omega\times C_T(\R^n)\subseteq[0,\infty)\times C_T(U\times\R^n),
\end{align*}
are also open in their respective spaces. Consider also the open set $\Omega^*_{U\times\R^n}\subseteq U\times\R^n$ given by
\[
 \Omega^*_{U\times\R^n} = \big\{ (p,q)\in U\times\R^n: (0,\cl{p},\cl{q})\in\Omega^*\big\}= \Omega_{U}\times\R^n.
\]

By Lemma \ref{lem:redDeg2} it follows that, if $\deg(\gamma,\Omega_U)$ is well defined and nonzero, then so is $\deg(G,\Omega^*_{U\times\R^n})$. Hence, Theorem \ref{th:g-single} applied to \eqref{eq:equiv2} yields the existence of a connected set, say $\Theta$, of nontrivial $T$-triples in $\Omega^*$ of \eqref{eq:equiv2} whose closure in $[0,\infty)\times C_T(U\times\R^n)$ meets $G^{-1}(0)\cap\Omega^*_{U\times\R^n}$ and is not contained in any compact subset of $\Omega^*_{U\times\R^n}$. By Lemma \ref{lem:homomTg}, we get that $\Gamma=\mathfrak{G}(\Theta)$ is a set of $T$-forced pairs of \eqref{eq:g-pertb-intro} in $\Omega$ since $\mathfrak{G}$ is a homeomorphism.
As above, observe that a $T$-triple $(0,x,y)$ of \eqref{eq:equiv2} is trivial if and only if $(0,x)$ is a trivial $T$-pair of \eqref{eq:g-pertb-intro}. The assertion follows.
\end{proof}

As already pointed out, this theorem cannot be deduced directly from Theorem~\ref{th:main1}. However, as in Remark \ref{re:extension1}, an extension of Theorem \ref{th:main2} to the following generalization of Equation \eqref{eq:g-pertb-intro} is possible:
\begin{equation}\label{eq:g-pertb-gen}
 \Big[\phi\Big(\lambda,x(t),x'(t)+\lambda k\big(t,x(t)\big)\Big)\Big]'\!\!=g\big(x(t),x'(t)\big)+\lambda f\big(t,x(t),x'(t),\lambda\big),\; \lambda\geq 0,
\end{equation}
with $\phi$, $f$ and $g$ as in \eqref{eq:0-pertb-intro} and $k\colon\R\times U\to\R^n$ Carath\'eodory and $T$-periodic in $t$.

\begin{remark}\label{re:extension2}
 The approach used for the proof of Theorem \ref{th:main2} can be applied also to equation \eqref{eq:g-pertb-gen}. In fact, the method used in Section \ref{se:rearrange} that brings to \eqref{eq:equiv2} takes \eqref{eq:0-pertb-gen} to the following equivalent system:
\begin{equation}\label{eq:equiv2-gen}
  \left\{\begin{array}{l}
         x'(t)=\phi_0^{-1}\big(y(t)\big)+\lambda H\big(\lambda,x(t),y(t)\big),\\
         y'(t)=g\Big(x(t),\phi_0^{-1}\big(y(t)\big)\Big)+\lambda F\big(t,x(t),y(t),\lambda\big),
         \end{array}
         \right.
 \end{equation}
where $F$ is given by \eqref{eq:bigF} and $H$ by \eqref{eq:bigH}. Define
the notions of $T$-forced pairs for \eqref{eq:g-pertb-gen} and $T$-triples of \eqref{eq:equiv2-gen} as the obvious parallels for the analogous concepts for \eqref{eq:g-pertb-intro} and \eqref{eq:equiv2}. Also, as in Lemma \ref{lem:homomTg} one has a homeomorphism between the set of $T$-forced pairs of \eqref{eq:g-pertb-gen} and the $T$-triples of \eqref{eq:equiv2-gen}.  Similarly to Remark \ref{re:extension1} the same argument of Theorem \ref{th:main2}, when $\deg(\gamma,\Omega_U)\neq 0$, implies the existence of a connected set $\Gamma$ of nontrivial $T$-forced pairs in $\Omega$ of \eqref{eq:g-pertb-gen} whose closure in $[0,\infty)\times C^1_T(U)$ intersects the set $\big\{ (0,\cl{p})\in [0,\infty) \times C^1_T(U): p\in\gamma^{-1}(0)\cap\Omega_U\big\}$ and is not contained in any compact subset of $\Omega$. In particular, when $U=\R^n$ and $\Omega=[0,\infty)\times C^1_T(\R^n)$ then $\Gamma$ is unbounded. 
\end{remark}

\section{Visual representation of branches} \label{sec:graphs}
Note that the connected sets of nontrivial $T$-forced pairs obtained in Theorems~\ref{th:main1} and~\ref{th:main2} are abstract, hard to visualize concepts, being naturally contained in $[0,\infty)\times C_T^1(\R^n)$.

In this section we outline a method to represent graphically the set of $T$-forced pairs of \eqref{eq:0-pertb-intro} and \eqref{eq:g-pertb-intro}. Our approach here, as in \cite{BiSp15} and \cite{CaPeSp23}, consists in obtaining a homeomorphic finite dimensional image $\Sigma$ of the set $\Gamma$ yielded by either of the above theorems, plot it when possible, and show a graph of some pertinent functions of the points of $\Gamma$; reasonable choices for second-order equations, as illustrated by the examples below, are the $C^1$-norm of the solution $x$ in any $T$-forced pair $(\lambda,x)$ and the diameter (measured with the Euclidean norm) of the orbit in the phase space, namely
\[
 \mathrm{diam}(x)=\max_{t,\tau\in[0,T]}\Big(\left|x(t)-x(\tau)\right|^2+\left|x'(t)-x'(\tau)\right|^2\Big)^{1/2}
 \]
Clearly, depending on the the feature of $\Gamma$ one wishes to highlight, one may elect to show other functions as, for instance, some projections of $\Sigma$.
 
Below, we assume $g$, $\phi$ and $f$ to be at least Lipschitz continuous, so that uniqueness of solutions and continuous dependence on initial data of \eqref{eq:0-pertb-intro} and \eqref{eq:g-pertb-intro} are granted. Furthermore, in order to keep this section reasonably short, we only focus on the scalar case ($n=1$) and on equations of the form \eqref{eq:g-pertb-intro}.

Consider, for $\lambda\geq 0$, the initial value problem
\begin{equation}\label{eq:iv-gprtb}
 \left\{\begin{array}{l}
         [\phi\big(\lambda,x(t),x'(t)\big)]'=g\big(x(t),x'(t)\big)+\lambda f\big(t,x(t),x'(t),\lambda\big),\\
         x(0)=p,\\
         x'(0)=v.
        \end{array}
\right. 
\end{equation}
and define the set
\[
 S=\left\{ (\lambda,p,v)\in[0,\infty)\times\R^{2}\;\;\left|\;\parbox{0.38\linewidth}{\eqref{eq:iv-gprtb} has a $T$-periodic solution}\right.\right\}
\]
The elements of $S$ are called \emph{starting points} for \eqref{eq:g-pertb-intro}. In other words, $S$ consists of those triples $(\lambda,p,v)$ such that $(p,v)$ are the initial conditions, at time $t=0$, of a $T$-periodic solution of \eqref{eq:g-pertb-intro} corresponding to $\lambda$. A starting point $(\lambda,p,v)$ is called \emph{trivial} when $\lambda=0$ and the solution of \eqref{eq:g-pertb-intro} starting a time $t=0$ from $(p,v)$ is constant. Clearly, $(0,p,v)$ is a trivial starting point if and only if $v=0$ and $g(p,0)=0$.

As in the previous section, denote by $\mathfrak{X}$ the set of $T$-forced pairs of \eqref{eq:g-pertb-intro}. By uniqueness and continuous dependence on initial data the map $\mathfrak{p}\colon \mathfrak{X}\to S$ given by
\[
(\lambda,x)\mapsto\big(\lambda,x(0),x'(0)\big)
\]
is a homeomorphism that establishes a correspondence between trivial $T$-forced pairs and trivial starting points. In other words, $\Sigma:=\mathfrak{p}(\Gamma)\subseteq [0,\infty)\times\R^{2}$ is the desired homemorphic image of $\Gamma$.
 
\begin{figure}[ht!]
 \begin{tabular}{@{}cc@{}}
\subfigure[$C^1$-norm of $x$ for $(\lambda,x)\in\Gamma$ and $\lambda\in {[0,1.45]}$ and diameter of the orbit in phase space.\label{fig:1a}]{\includegraphics[width=0.45\linewidth]{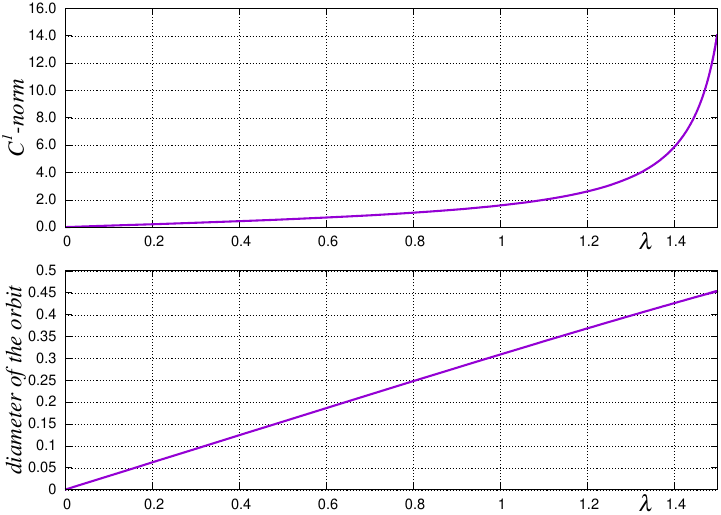}} &
 \subfigure[Starting points $(\lambda,p,v)\in\Sigma$ for $\lambda\in {[0,1.45]}$.\label{fig:1b}]{\includegraphics[width=0.53\linewidth]{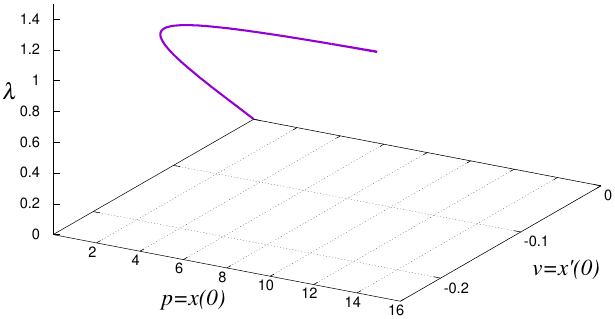}}
 \end{tabular}
\caption{Representation of a portion of $\Gamma$ of Example \ref{ex:1}}\label{fig:1}
\end{figure}

\begin{example}\label{ex:1}
Consider the following scalar equation:
\begin{equation}\label{eq:example1}
 \big[x'(t)^3+x'(t)\big]'=\arctan\big(x(t)\big)+ \lambda \big(\sin(2\pi t)-1\big),
\end{equation}
with $\lambda\geq0$. Here, $T=1$, $g(p,v)=\arctan(p)$, $\phi(\lambda,p,v)=v^3+v$ and $f(t,p,v)=\sin(2\pi t)-1$, so that
$\gamma(p)=g(p,0)=\arctan(p)$. Take $U=\R$ and $\Omega=[0,\infty)\times C^1_T(\R)$, and observe that $\deg(\gamma,\Omega_U)=\deg(\gamma,\R)=1$. Thus, Theorem \ref{th:g-single} yields an unbounded connected set $\Gamma$ of nontrivial $T$-forced pairs whose closure in $[0,\infty)\times C^1_T(\R)$ contains $(0,\cl{0})$.

Figure \ref{fig:1a} shows the $C^1$-norm and orbit diameter in the phase plane of the solutions of the $T$-forced pairs of \eqref{eq:example1} when $\lambda$ varies in $[0,1]$. Figure \ref{fig:1b}, instead, shows the starting points of $\Sigma$ for $\lambda$ in the longer interval $[0,1.6]$. Taken together, the images in Figure \ref{fig:1} seem to suggest that, as $\lambda$ grows, the initial point moves away rapidly from the origin whereas the speed and diameter of the orbit do not grow as fast.
\end{example}

\begin{figure}[ht!]
 \begin{tabular}{@{}cc@{}}
\subfigure[$C^1$-norm of $x$ for $(\lambda,x)\in\Gamma_0\cup\Gamma_1$ and diameter of the orbit in phase space.\label{fig:2a}]{\includegraphics[width=0.45\linewidth]{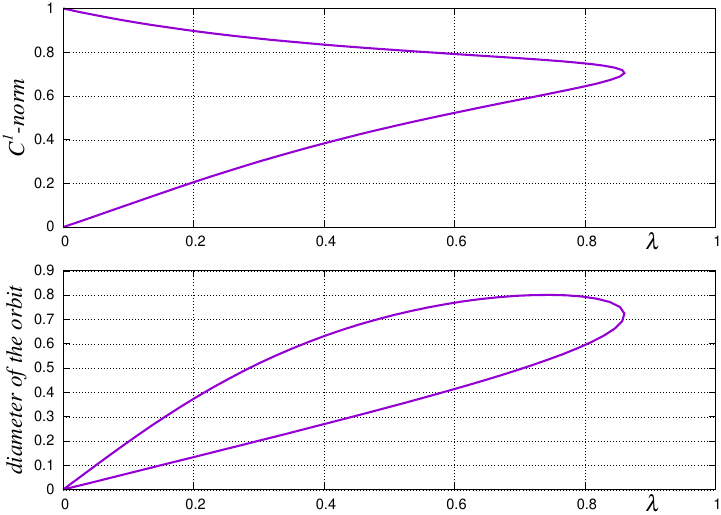}} &
 \subfigure[Starting points $(\lambda,p,v)\in\Sigma_0\cup\Sigma_1$.\label{fig:2b}]{\includegraphics[width=0.53\linewidth]{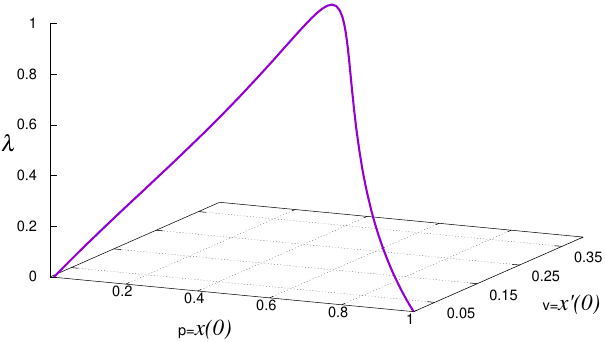}}
 \end{tabular}
\caption{Representation of $\Gamma_0$ and $\Gamma_1$ of Example \ref{ex:2}}\label{fig:2}
\end{figure}

In the following example the behavior of the connected set of $T$-forced pairs is quite different.

\begin{example}\label{ex:2}
 Consider the following scalar equation:
\begin{equation}\label{eq:example2}
 \big[x'(t)^3+2x'(t)\big]'=x(t)^2-x(t)- \lambda \big(\sin(t)-2x(t)^2\big),
\end{equation}
with $\lambda\geq0$. Here $\phi(\lambda,p,v)=v^3+2v$, $T=2\pi$, $g(p,v)=p^2-p$ and $f(t,p,v)=2p^2-\sin (t)$; so that $\gamma(p)=g(p,0)=p^2-2p$. Let $U_0=\R\setminus\{0\}$ and $U_1=\R\setminus\{1\}$ and consider the following open subsets of $[0,\infty)\times C^1_T(\R)$:
\begin{equation*}
 \mathcal{A} =[0,\infty)\times C^1_T(U_0),\qquad
 \mathcal{B} =[0,\infty)\times C^1_T(U_1).
\end{equation*}
 Since $\deg(\gamma,\mathcal{A}_{U_0})=-1$ and $\deg(\gamma,\mathcal{B}_{U_1})=1$, Theorem \ref{th:g-single} yields two connected sets $\Gamma_0$ and $\Gamma_1$ of nontrivial $T$-forced pairs of \eqref{eq:example2} whose closure in $[0,\infty)\times C^1_T(U_0)$ and $[0,\infty)\times C^1_T(U_1)$ contains $(0,\cl{0})$ and $(0,\cl{1})$, respectively. Here, as above, $\cl{0}$ and $\cl{1}$ denote, respectively, the constant functions $t\mapsto 0$ and $t\mapsto 1$. Let $\Sigma_0=\mathfrak{p}(\Gamma_0)$ and $\Sigma_1=\mathfrak{p}(\Gamma_1)$.
 In Figure \ref{fig:2} we plot, again, the $C^1$-norm and orbit diameter of the solutions and the starting points.

 Figure \ref{fig:2} suggests that $\Sigma_0=\Sigma_1$, so that $\Gamma_0=\Gamma_1$ and that they are bounded. This is not in contradiction with Theorem \ref{th:g-single} as $\deg(\gamma,\R\setminus\{0,1\})=0$.
 
 Observe also that Figure \ref{fig:2} suggests that there exists $\lambda_0\in (0.8,1)$ such that \eqref{eq:example2} admits two $T$-periodic solutions for $\lambda \in(0,\lambda_0)$ and has no $T$-periodic solutions for $\lambda\in (\lambda_0,1)$.
\end{example}

In the following example the set of $T$-forced pairs contains two unbounded connected subsets. The numerical experiment indicates that their qualitative behavior is different with respect to the projection along the $\lambda$ axis.

\begin{figure}[ht!]
 \begin{tabular}{@{}cc@{}}
\subfigure[$C^1$-norm of $x$ for $(\lambda,x)\in\Gamma_{-1}$ and diameter of the orbit in phase space.\label{fig:3a}]{\includegraphics[width=0.48\linewidth]{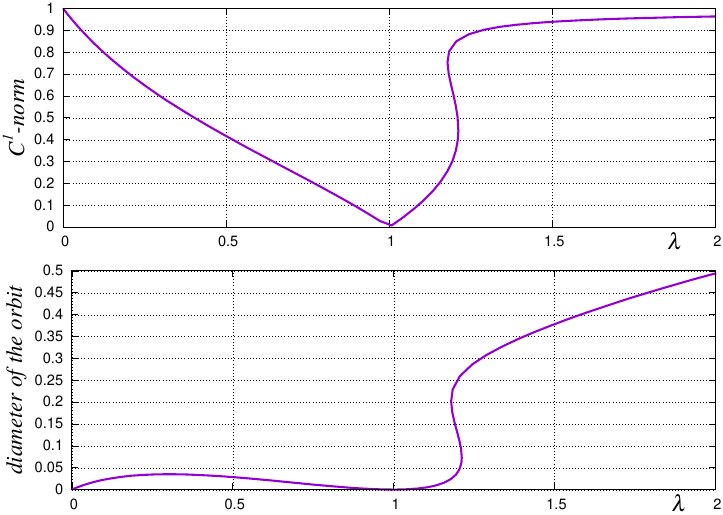}} &
 \subfigure[$C^1$-norm of $x$ for $(\lambda,x)\in\Gamma_{+1}$ and diameter of the orbit in phase space.\label{fig:3b}]{\includegraphics[width=0.48\linewidth]{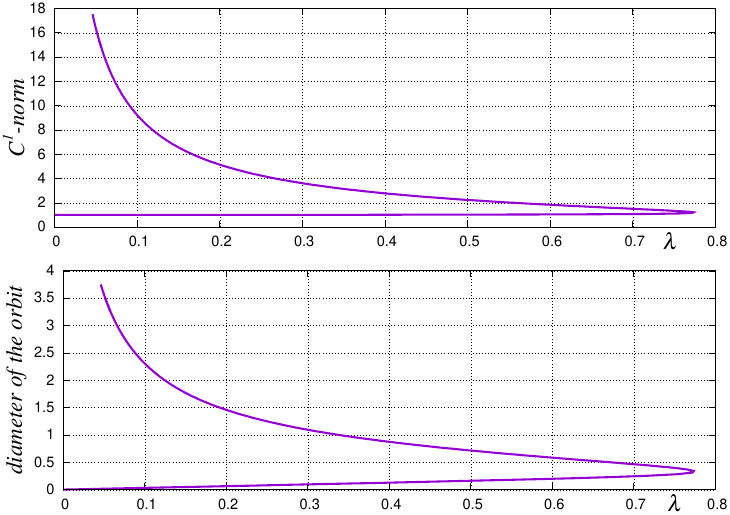}}
 \end{tabular}
\caption{$C^1$-norm and orbit diameter of points in $\Gamma_{-1}$ and $\Gamma_1$ of Example \ref{ex:3}}\label{fig:3}
\end{figure}

\begin{figure}[h!]
 \includegraphics[width=0.85\linewidth]{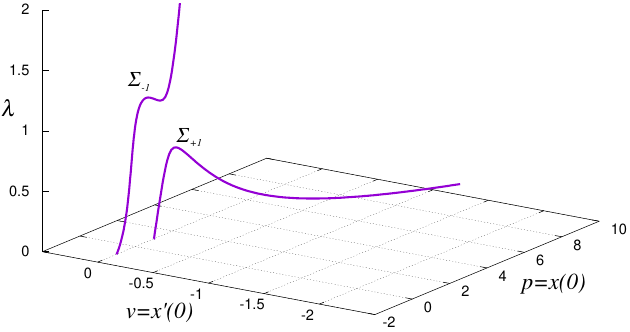}
 \caption{A portion of the set of starting points $(\lambda,p,v)\in\Sigma_{-1}\cup\Sigma_{+1}$ of equation \eqref{eq:example3} in Example \ref{ex:3}.\label{fig:4}}
\end{figure}

\begin{example}\label{ex:3}
  Consider the following scalar equation:
\begin{equation}\label{eq:example3}
 \big[\lambda x'(t)^3+x'(t)\big]'=\frac{x(t)^2-1}{x(t)^2+1}+ \lambda \big(x(t)^2\sin(2\pi t)+1-x(t)\big),
\end{equation}
with $\lambda\geq0$. In this example $\phi(\lambda,p,v)=\lambda v^3+v$, $T=1$, $g(p,v)=\frac{p^2-1}{p^2+1}$ and $f(t,p,v)=p^2\sin(2\pi t)+1-p$. Let $U_{-1}=\R\setminus\{-1\}$ and $U_1=\R\setminus\{1\}$ and, as in Example \ref{ex:2}, consider the following open subsets of $[0,\infty)\times C^1_T(\R)$:
\begin{equation*}
 \mathcal{A} =[0,\infty)\times C^1_T(U_{-1}),\qquad
 \mathcal{B} =[0,\infty)\times C^1_T(U_1).
\end{equation*}
We have that $\deg(\gamma,\mathcal{A}_{U_{-1}})=-1$ and $\deg(\gamma,\mathcal{B}_{U_1})=1$. Hence, Theorem \ref{th:g-single} yields two connected sets $\Gamma_{-1}$ and $\Gamma_{+1}$ of nontrivial $T$-forced pairs whose closure in $[0,\infty)\times C^1_T(U_{-1})$ and $[0,\infty)\times C^1_T(U_{+1})$ contains $(0,\cl{-1})$ and $(0,\cl{1})$, respectively. The figures \ref{fig:3} and \ref{fig:4} seem to suggest that $\Gamma_{-1}$ and $\Gamma_{+1}$ do not intersect and are unbounded. In particular, it seems that the set $\{\lambda:(\lambda,x)\in\Gamma_{+1}\}$ is bounded from above. Figure \ref{fig:4}, in particular shows a portion (in particular with $\lambda\leq 2$) of the sets of starting points in $\Sigma_{-1}=\mathfrak{p}(\Gamma_{-1})$ and $\Sigma_{+1}=\mathfrak{p}(\Gamma_{+1})$.

Notice that according to the figures \ref{fig:3} and \ref{fig:4}, there are intervals $I_1\subset (0,0.8)$ and $I_2\subset(1,1.5)$, such that when $\lambda\in I_1$, Equation \eqref{eq:example3} admits three $T$-periodic solutions, one whose starting point lies in $\Sigma_{-1}$ and two in $\Sigma_{+1}$, and for $\lambda\in I_2$, there are two $T$-periodic solutions, both with starting points in $\Sigma_{-1}$. There is just one $T$-periodic solution for  $\lambda\in [0,\infty)\setminus I_1\cup I_2$, and has starting point in $\Sigma_{-1}$.   
\end{example}

\begin{remark}\label{re:finale}
 The examples above may have given the false impression that the set $\Sigma=\mathfrak{p}(\Gamma)$ is always a ``nice'' smooth $1$-dimensional manifold. Although this can be true quite often, it is false in general as the set of starting points may look rather ``wild''. Consider, for instance, the following equation:
 \begin{equation}\label{eq:example4}
 \big[\frac{1}{3}x'(t)^3+2x'(t)\big]'=x(t)^2-x(t)+ \lambda \big(\sin(2\pi t)-x(t)^2\big),\quad\lambda\geq 0,
 \end{equation}
which is not very different from \eqref{eq:example2} but with $T=1$. Figure \ref{fig:5} shows a portion of the set of starting points of \eqref{eq:example4}.
\end{remark}

\begin{figure}[h!]
 \includegraphics[width=0.85\linewidth]{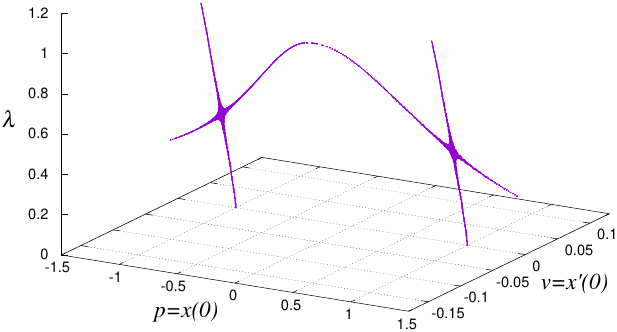}
 \caption{A portion of the set of starting points $(\lambda,p,v)\in\Sigma\cap\big([-1.5,1.5]\times[-0.2,0.15]\times[0,1.1]\big)$ of equation \eqref{eq:example4} in Remark \ref{re:finale}.\label{fig:5}}
\end{figure}
 
 We finally observe that, as suggested by the numerical examples presented above, the number of $T$-periodic solutions associated to particular choices of the parameter may vary in a nontrivial way with $\lambda$ and it may be subject to abrupt transitions when $\lambda$ crosses specific values.


\begin{thebibliography}{99}


\bibitem{Arn2006}
Arnold, Vladimir I. \textit{Ordinary differential equations}, {Universitext}, {Springer, Berlin}, {2006}.


\bibitem{BGKT}
Benedikt, Ji\v{r}\'{\i}; Girg, Petr; Kotrla, Luk\'{a}\v{s}; Tak\'{a}\v{c}, Peter.
Origin of the $p$-Laplacian and A.\ Missbach.
{\it Electron.\ J.\ Differ.\ Equ.\ } \textbf{2018}, Paper No. 16, 17 pp.

\bibitem{BeCa08}
Benevieri, Pierluigi; Calamai, Alessandro.
Bifurcation results for a class of perturbed Fredholm maps,
{\it Fixed Point Theory Appl.\ } \textbf{2008}, Article ID 752657, 19 pp. 

\bibitem{BeFe24}
Benevieri, Pierluigi; Feltrin, Guglielmo.
Atypical bifurcation for periodic solutions of $\phi$-Laplacian systems.
Preprint, arXiv:2406.00325 (2024).

\bibitem{BFPS03} Benevieri Pierluigi; Furi, Massimo; Pera, Maria Patrizia; Spadini, Marco.  
{\em An Introduction to Topological Degree in Euclidean Spaces}.  Technical Report 
n.~42, Gennaio 2003, Universit\`a di Firenze, Dipartimento di Matematica Applicata, 
2003.

\bibitem{bm08}
Bereanu, Cristian; Mawhin, Jean.
Periodic solutions of nonlinear
perturbations of $\Phi$-La\-plac\-ians with possibly bounded $\Phi$,
{\it Nonlin.\ Analysis \textsc{tma}} \textbf{68}, (2008), 1668--1681.

\bibitem{BiCaPa19}
Biagi, Stefano; Calamai, Alessandro; Papalini, Francesca.
Heteroclinic solutions for a class of boundary value problems associated with singular equations,
{\it Nonlin.\ Analysis \textsc{tma}} \textbf{184} (2019), 44--68.

\bibitem{BiSp15}
 Bisconti Luca, Spadini Marco,
About the notion of non-{$T$}-resonance and applications to
 topological multiplicity results for {ODE}s on differentiable manifolds,
\textit{Math. Methods Appl. Sci.} {\bf 38} (2015), no.~18, 4760--4773.

\bibitem{bicasp}
Bisconti, Luca; Calamai, Alessandro; Spadini, Marco.
Periodic solutions of semi-explicit differential-algebraic equations with time-dependent constraints,
{\it Bound.\ Value Probl.\ } {\bf 2014} 2014:179, 19 pp.

\bibitem{BogRon}
Bogn\'ar, G.; Ront\'o, M.
Numerical-analytic investigation of the radially symmetric solutions for some nonlinear PDEs,
{\it Comput.\ Math.\ Appl.\ } \textbf{50} (2005), No. 7, 983--991.

\bibitem{BFZ21}
Boscaggin, Alberto; Feltrin, Guglielmo; Zanolin, Fabio.
Uniqueness of positive solutions for boundary value problems associated with indefinite Laplacian-type equations.
{\it Open Math.\ } \textbf{19} (2021), 163--183. 

\bibitem{cab11}  Cabada, Alberto.
An overview of the lower and upper solutions method with nonlinear boundary value conditions.
{\it Bound.\ Value Probl.\ } \textbf{2011}, Art. ID 893753, 18 pp. 

\bibitem{Ca2011B}
Calamai, Alessandro.
Branches of harmonic solutions for a class of periodic
  differential-algebraic equations,
{\it Commun.\ Appl.\ Anal.\ } \textbf{15} (2011), no.~2,3 and 4, 273--282.

\bibitem{Ca2011H}
Calamai, Alessandro.
Heteroclinic solutions of boundary value problems on the real line
      involving singular $\Phi$-Laplacian operators,
{\it J.\ Math.\ Anal.\ Appl.\ } \textbf{378} (2011), no.~2, 667--679.

\bibitem{CaPeSp23} {Calamai, Alessandro and Pera, Maria Patrizia and Spadini, Marco}, {Periodic perturbations of reducible scalar second order functional differential equations}, {Electron. J. Qual. Theory Differ. Equ.}, {2023}, {Paper No. 18, 23}, {1417-3875}
              

\bibitem{capesp-libro}
Calamai, Alessandro; Pera, Maria Patrizia; Spadini, Marco.
Branches of forced oscillations for a class of implicit equations involving the $\Phi$-Laplacian, in:
\textit{Topological Methods for Delay and Ordinary Differential Equations -- With Applications to Continuum Mechanics}, P.\ Amster and P.\ Benevieri (Eds.), Birkh\"auser, Cham, p. 151--166.

\bibitem{CaSp24}  Calamai, Alessandro; Spadini, Marco. Carathéodory Periodic Perturbations Of Degenerate Systems
(2024), Electronic Journal of Differential Equations, 2024, art. no. 39.

\bibitem{CaSp2} Calamai, Alessandro; Spadini, Marco.
 Periodic perturbations of constrained motion problems on a class of implicitly defined manifolds,
{\it Commun.\ Contemp.\ Math.\ } \textbf{17} (2015), no.~2, 1450027, 19~pp.

\bibitem{Cumapa2011}
Cupini, Giovanni; Marcelli, Cristina; Papalini, Francesca.
Heteroclinic solutions of boundary-value problems on the real line involving general nonlinear differential operators,
{\it Differ.\ Integral Equ.\ } \textbf{24} (2011), no.~7-8, 619--644. 

\bibitem{DiMa21}
Dinca, George; Mawhin, Jean. {\it Brouwer degree -- the core of nonlinear analysis}. Progress in Nonlinear Differential Equations and their Applications, Birkh\"{a}user/Springer, Cham (2021).

\bibitem{Frigon2013}
El Khattabi, Noha; Frigon, Marlene; Ayyadi, Nourredine.
Multiple solutions of boundary value problems with $\phi$-Laplacian operators and under a Wintner-Nagumo growth condition.
{\it Bound.\ Value Probl.\ } \textbf{2013} (2013), Paper No. 236, 21 pp.



\bibitem{FSZ19}
Feltrin, Guglielmo; Sovrano, Elisa; Zanolin, Fabio.
Periodic solutions to parameter-dependent equations with a $\phi$-Laplacian type operator,
{\it NoDEA, Nonlinear Differ.\ Equ.\ Appl.\ } \textbf{26} (2019), no.~5, Paper no. 38, 27 pp.

\bibitem{FZ21}
Feltrin, Guglielmo; Zanolin, Fabio.
Bound sets for a class of $\phi$-Laplacian operators.
{\it J.\ Differ.\ Equations} \textbf{297} (2021), 508--535. 



\bibitem{FuPe1997} Furi Massimo; Pera Maria Patrizia.
Global branches of harmonic solutions to periodic ODEs on manifolds, {\it Boll.\ Un.\ Mat.\ Ital.\ } {\bf 11-A}
(1997), 709--722.

\bibitem{FP97-car}  Furi, Massimo; Pera, Maria Patrizia.
Carath\'eodory periodic perturbations of the zero vector field on manifolds,  \textit{Topological methods in nonlinear analysis} {\bf 10} (1997), n. 1, 79--92.

\bibitem{FuPeSp10} Furi Massimo; Pera Maria Patrizia; Spadini Marco.
A set of axioms for the degree of a tangent vector field on differentiable manifolds. {\it Fixed Point Theory Appl.} {\bf 2010}, Art.\ ID 845631, 11~pp.

\bibitem{FuSp1997} Furi Massimo; Spadini Marco.
On the set of harmonic
solutions to periodically perturbed autonomous differential equations on
manifolds, {\it Nonlin.\ Analysis \textsc{tma}} \textbf{29} (1997), 963--470.

  
\bibitem{GyP}  Guillemin, Victor; Pollack, Alan. \textit{Differential Topology,}
Prentice-Hall Inc., Englewood Cliffs, New Jersey, 1974.

\bibitem{Hir76}  Hirsch, Morris~W.
 \textit{Differential Topology}, Graduate
Texts in  Mathematics, Vol.\ 33, Springer, Berlin 1976.

\bibitem{L78}  Lloyd, N.G. {\em Degree theory}. Cambridge University Press, Cambridge, 1978.

\bibitem{MilTop}  Milnor, John W. \textit{Topology from the differentiable
viewpoint,} Univ. press of Virginia, Charlottesville, 1965.

\bibitem{Ni74} Nirenberg, Louis. {\it Topics in Nonlinear Functional Analysis}. Courant Institute of Mathematical Sciences, New York, 1974.

\bibitem{PRRFB}
Picasso, Marco; Rappaz, Jacques; Reist, Adrian; Funk, Martin; Blatter, Heinz.
Numerical simulation of the motion of a two-dimensional glacier.
{\it Int.\ J.\ Numer.\ Methods Eng.\ } \textbf{60} (2004), No. 5, 995--1009.

\bibitem{RT05}
Rach\r unkov\'a, Irena; Tvrd\' y, Milan.
Periodic problems with $\phi$-Laplacian involving non-ordered lower and upper functions.
{\it Fixed Point Theory} \textbf{6} (2005), no.~1, 99--112.

\bibitem{Spa00}  Spadini, Marco.
Harmonic solutions of periodic Carath\'eodory perturbations of autonomous ODE's on manifolds, {\it Nonlin.\ Analysis \textsc{tma}} {\bf 41A} (2000), 477--487.

\bibitem{Spa2006}
Spadini, Marco.
Branches of harmonic solutions to periodically perturbed coupled differential equations on manifolds.
{\it Discrete Contin. Dyn. Syst.} \textbf{15} (2006), no.~3, 951--964.
  
\bibitem{Spa2010}
Spadini, Marco.
A note on topological methods for a class of differential-algebraic equations,
{\it Nonlin.\ Analysis \textsc{tma}} \textbf{73} (2010), no.~4, 1065--1076.


 \end{thebibliography}
\end{document}